\documentclass[11pt]{article}

\usepackage[margin=1.27in]{geometry}

\usepackage[misc]{ifsym} 
\usepackage{lipsum}
\usepackage{amssymb}
\usepackage{amsmath}
\usepackage{latexsym}
\usepackage{bm}
\usepackage{enumitem}
\usepackage{graphicx}
\usepackage{amscd}
\usepackage{extarrows}
\usepackage{amsfonts}
\usepackage{amssymb}
\usepackage{indentfirst}
\usepackage{bbm}
\usepackage{mathdots}
\usepackage{mathtools}
\usepackage[linktocpage]{hyperref}
\usepackage{amsthm}
\hypersetup{
    colorlinks,
    citecolor=black,
    filecolor=black,
    linkcolor=blue,
    urlcolor=black
}

\usepackage{quiver}
\usepackage[multiple]{footmisc}
\usepackage[title]{appendix}

\usepackage{float}

\theoremstyle{plain}
\newtheorem{theorem}{Theorem}[section]
\newtheorem{proposition}[theorem]{Proposition}
\newtheorem{lemma}[theorem]{Lemma}
\newtheorem{corollary}[theorem]{Corollary}
\newtheorem*{theorem*}{Theorem}

\theoremstyle{remark}
\newtheorem{remark}[theorem]{Remark}

\theoremstyle{definition}

\DeclareMathOperator{\tr}{tr}

\DeclareMathOperator{\vol}{vol}
\DeclareMathOperator{\cpt}{cpt}

\DeclareMathOperator{\SL}{SL}
\DeclareMathOperator{\GL}{GL}
\DeclareMathOperator{\PGL}{PGL}
\DeclareMathOperator{\PSL}{PSL}

\DeclareMathOperator{\St}{St}
\DeclareMathOperator{\SO}{SO}

\DeclareMathOperator{\JL}{JL}
\DeclareMathOperator{\SU}{SU}

\DeclareMathOperator{\PD}{PD}

\DeclarePairedDelimiter\abs{\lvert}{\rvert}

\title{Von Neumann Dimensions and Trace Formulas II:\\
A Jacquet-Langlands correspondence for Arithmetic Group Algebras in $\GL(2)$}
\author{Jun Yang}
\date{}

\begin{document}
\maketitle
\begin{abstract}
We propose a global Jacquet-Langlands correspondence for the modules over the von Neumann algebras of $S$-arithmetic subgroups of $\GL(2)$ and of a quaternion algebra $D$, which are both defined over a totally real number field $F$. 
If a representation $\pi'=\otimes\pi'_v$ of $D^{\times}(\mathbb{A}_F)$ corresponds to a representation $\pi=\otimes \pi_v$ of $\GL(2,\mathbb{A}_F)$, 
we have
\begin{equation*}
\frac{\dim_{\mathcal{L}(\SL(2,\mathcal{O}_S))}\pi_S}{\dim_{\mathbb{C}}\pi'_S}=\abs*{ \frac{\zeta_{D}(0)}{\zeta_{F}(0)}},
\end{equation*}
where $\zeta_F,\zeta_{D}$ are the zeta functions of $F,D$ respectively. 
\end{abstract}

\section{Introduction}

In the 1960s, H. Jacquet and R. Langlands \cite{JL70} proved a correspondence between the automorphic representations of the multiplicative group $D^{\times}$ of a quaternion algebra $D$ over a global field $F$ and certain cuspidal automorphic representation of $\GL(2)$. 
They also gave the local correspondence between
the square-integrable representations of $\GL(2,F_v)$ and the admissible irreducible representations of the multiplicative group $D^{\times}(F_v)$.  
The local results have been generalized to $\GL(n)$ and higher dimensional division algebras by Deligne-Kazhdan-Vign\'{e}ras\cite{DKV84} and J. Rogwaski\cite{Rgw83} while the global one has been generalized later by A. Badulescu \cite{Badu08}. 
One of the key techniques is to show that for certain functions $\varphi,\varphi'$ on $\GL(2,\mathbb{A}_F)$ and $D^{\times}(\mathbb{A}_F)$ that matches appropriately, the traces of the actions of $\varphi,\varphi'$ on the cuspidal spectrum of $G=\GL(2)$ and the entire automorphic spectrum of $D^{\times}$ (excludes $1$-dimensional ones) coincide, i.e.,
\begin{equation}\label{eintrotf}
\tr(\phi|_{L_{\rm{cusp}}^2(Z(\mathbb{A})G(F)\backslash G(\mathbb{A}),\omega)})=\tr(\phi'|_{L_0^2(Z(\mathbb{A})D^{\times}(F)\backslash D^{\times}(\mathbb{A}),\omega)}),
\end{equation}
for any character $\omega$ of $F^{\times}\backslash \mathbb{A}^{\times}$. 
As one term of these traces, it is also known that, if $\varphi$ is matched to $\varphi'$, we have $\vol(Z(\mathbb{A})G(F)\backslash G(\mathbb{A}))\cdot \phi(1)=\vol(Z(\mathbb{A})D^{\times}(F)\backslash D^{\times}(\mathbb{A}))\cdot \phi'(1)$. 
Finally, it gives a well-defined correspondence $\pi'=\otimes \pi'_v\longmapsto \pi=\otimes \pi_v$.

Moreover, for each place $v$ in the ramified set $S$ of $D$, the dimension and formal degree are proportional, i.e., $d(\pi_v)/\dim \pi'_v=d(\rho_v)/\dim \rho'_v$, if $\pi_v'\mapsto \pi_v$ and $ \rho_v'\mapsto \rho_v$,  
where $d(-)$ denotes the formal degree of a square-integrable representation. 
This property is known to be true for $\GL(n)$ and higher dimensional division algebras (see \cite{Rgw83} and \cite{CMS90}). 
Thus we obtain
\begin{equation}\label{eintrofdim2}
     \vol(Z(\mathbb{A})G(F)\backslash G(\mathbb{A}))\cdot \prod_{v\in S}d(\pi_v)=\vol(Z(\mathbb{A})D^{\times}(F)\backslash D^{\times}(\mathbb{A}))\cdot \prod_{v\in S}d(\pi'_v).
\end{equation}

As an analogy of Equation \ref{eintrofdim2},  there is a formula about von Neumann dimensions:  
Given a lattice $\Gamma$ in a unimodular group $G$, a square-integrable representation $H$ (of $G$) is naturally a module over the group von Neumann algebra $\mathcal{L}(\Gamma)$ of $\Gamma$, which leads to the real-valued von Neumann dimension $\dim_{\mathcal{L}(\Gamma)}H$. 
If we take $G$ to be a connected semi-simple real Lie group with a square-integrable representation $(\pi,H)$, Atiyah-Schmid \cite{AS77} proved 
\begin{equation}\label{eintroAS}
    \dim_{\mathcal{L}(\Gamma)}H=\vol(\Gamma\backslash G)\cdot d(\pi),
\end{equation}
which says the von Neumann dimension does not depend on the choice of Haar measure since the two dependencies cancel out on the right-hand side. 
Thus Equations \ref{eintrofdim2} and 
\ref{eintroAS} may predict a correspondence for the $\mathcal{L}(\Gamma)$-modules 
for an arithmetic group $\Gamma$. 

One interesting property is that the equivalence classes of $\mathcal{L}(\Gamma)$-modules can be given by their von Neumann dimensions: if $\Gamma$ has the infinite conjugacy class property (e.g., $\Gamma=\PSL(2,\mathbb{Z})$, see Section \ref{svna} for the definition), this dimension determines the equivalence class of $\mathcal{L}(\Gamma)$-modules: $H_1\cong H_2$ as $\mathcal{L}(\Gamma)$-modules if and only if $\dim_{\mathcal{L}(\Gamma)}H_1=\dim_{\mathcal{L}(\Gamma)}H_2$. 
Thus it is reasonable to consider the correspondence between these dimensions over $\mathcal{L}(\Gamma),\mathcal{L}(\Gamma')$ for some arithmetic subgroups $\Gamma,\Gamma'$ in $\GL(2)$ and $D^{\times}$. 


The term $\prod_{v\in S}d(\pi_v)$ in Equation \ref{eintrofdim2} is the formal degree for the group $\Pi_{v\in S}\GL(2,F_v)$.  
Thus it is reasonable to consider the pairs of $S$-arithmetic subgroups $\PGL(2,\mathcal{O}_S)$ with $\PD^{\times }(\mathcal{O}_S)$, and similarly $\SL(2,\mathcal{O}_S)$ with $D^{1}(\mathcal{O}_S)$. 
Note both $\PD^{\times }(\mathcal{O}_S)$ and $D^{1}(\mathcal{O}_S)$ are finite groups. 

Consider a totally real number field $F$. 
Let $\zeta_F,\zeta_{D}$ be the (Dedekind) zeta function of $F$ and $D$ respectively. 
We assume $D$ ramifies at each real place of $F$. 
\begin{theorem*}
Given a global Jacquet-Langlands correspondence   $\pi'=\otimes\pi'_v$ $\mapsto$ $\pi=\otimes \pi_v$, 
we have
\begin{equation*}
\frac{\dim_{\mathcal{L}(\SL(2,\mathcal{O}_S))}\pi_S}{\dim_{\mathcal{L}(D^{1}(\mathcal{O}_S))}\pi'_S}=\abs*{\frac{\zeta_{D}(0)}{\zeta_{F}(0)}} \cdot |D^{1}(\mathcal{O}_S)|. 
\end{equation*}
\end{theorem*}
 
This gives a correspondence between the infinite dimensional module $\pi_S$ over the type $\text{II}_1$ algebra $\mathcal{L}(\SL(2,\mathcal{O}_S))$ and the finite dimensional module $\pi'_S$ over the type $\text{I}$ algebra $\mathcal{L}(D^{1}(\mathcal{O}_S))$. 
Let us outline its proof: Section \ref{spreJL} and \ref{svna} are provided with preliminaries and an emphasis on $S$-adelic groups and $S$-arithmetic groups, where some results are known.  
In Section \ref{sdim}, we consider the matching of modules over the algebras of certain pairs of arithmetic groups. 
We start with the Steinberg representations and give the explicit formulas for the dimensions.

\section{Trace formulas of $\GL(2)$ and the correspondence for $S$-adelic groups}\label{spreJL}

Here we have a brief review of the Jacquet-Langlands correspondence of $\GL(2)$ over a number field. 
The matching of certain terms of the trace formulas is also needed for the proof of the  results in the following sections. 
We fix the following notations.
\begin{itemize}
    \item $F$ is a totally real number field;
    \item $V,V_{\infty},V_f$ are the set of places, infinite places and finite places of $F$ respectively and $F_v$ is the local field at $v$;
    \item $D$ is a quaternion algebra over $F$;
    \item $S$ is the ramification set of $D$, i.e., $D^{\times}(F_v)\cong \GL(2,F_v)$ if and only if $v\in S$;
    \item $G=\GL(2)$ and $G'=D^{\times}$ with the isomorphic centers $Z\cong Z'$;
    \item $\overline{G}=Z\backslash G=\PGL(2)$ and $\overline{G'}=Z'\backslash D^{\times}=\PD^{\times}$
\end{itemize}
It is known that $S$ is of even cardinality \cite[Theorem 7.3.6]{McRd219}. 
Now we fix a character $\omega$ of $F^{\times}\backslash \mathbb{A}^{\times}$ and let
\begin{itemize}
\item $\Pi(G'_v)=$ the set of equivalence classes of irreducible admissible representations of $G'_v$;
\item $\Pi^2(G_v)=$ the set of classes of square-integrable representations\footnote{Here we usually consider the representations of a topological group $G$ whose matrix coefficents are in $L^2(G/Z(G))$ if the center $Z(G)$ is not compact. 
If $Z(G)$ is not compact, there are no square-integrable representations in $L^(G)$. 
If $(\pi,H)$ is such a representation, we take nonzero $u,v\in H\subset L^2(G)$ and obtain $\langle u,v\rangle=\int_G u(g)\overline{v(g)}dg=\int_{G/ Z(G)}u(g)\overline{v(g)}\left(\int_{Z(G)}1dz\right)dg$, where the inner integral diverges.}  of $G_v$;

\item $\mathcal{A}^S_0(G)=$ the set of classes of irreducible representation $\pi=\bigotimes_{v\in V} \pi_v$ of $G(\mathbb{A})$ in $L_{\rm{cusp}}^2(Z(\mathbb{A})G(F)\backslash G(\mathbb{A}),\omega)$ such that $\pi_v$ is a discrete series for each $v\in S$; 
\item $\mathcal{A}_0(G')=$ the set of classes of irreducible representation $\pi'=\bigotimes_{v\in V} \pi'_v$ of $D^{\times}(\mathbb{A})$ in $L^2(Z(\mathbb{A})D^{\times}(F)\backslash D^{\times}(\mathbb{A}),\omega)$ which are not one-dimensional. 
\end{itemize}
Recall that for regular semisimple elements $t,t'$ in $G_v,G'_v$, we write $t\sim t'$ if they have the same eigenvalues in $F_v$. 
For hyperbolic $t\in G_v$, there is no $t'\in G'_v$ such that $t\sim t'$. 

Following \cite[\S 16]{JL70} and \cite[\S 8]{GJ79}, we state the local and global Jacquet-Langlands correspondences for $\GL(2)$ as follows, where $\theta_{\pi}$ denotes the distribution character of $\pi$ (see \cite{Delorme96}). 
\begin{theorem}\label{tJL}
\begin{enumerate}
    \item There is a bijection $\JL_v\colon \Pi(G'_v)\to \Pi^2(G_v)$ such that
    \begin{center}
        $\theta_{\pi'_v}(t')=-\theta_{\pi_v}(t)$, 
    \end{center} 
    if $\JL_v(\pi'_v)=\pi_v$ and $t'\sim t$. 
    \item There is a bijection $\JL\colon \mathcal{A}_0(G')\to \mathcal{A}^S_0(G)$ such that if $\pi'=\bigotimes\limits_{v\in V}\pi'_v$ $\mapsto$ $\pi=\bigotimes\limits_{v\in V} \pi_v$, 
then $\JL_v(\pi'_v)=\pi_v$ for $v\in S$ and $\pi'_v\cong \pi_v$ for $v\notin S$. 
\end{enumerate}
\end{theorem}

We describe the matching orbital integrals, which is necessary for both the proof of Theorem \ref{tJL} and the rest parts of this paper. 
Let $C^{\infty}_{\cpt}(G_v),C^{\infty}_{\cpt}(G'_v)$ be the complex algebra of smooth functions on $G_v,G'_v$ with compact supports. 
Let $\varphi_v\in C^{\infty}_{\cpt}(G_v)$ and $\varphi'_v\in C^{\infty}_{\cpt}(G'_v)$. 
For elliptic elements $t,t'$ in $G_v,G'_v$,  let $\overline{G}_v(t)$ and $\overline{G}'_v(t')$ be their centralizers and define 
\begin{center}
$\Phi(\varphi_v,t)=\int_{\overline{Z(t)}\backslash \overline{G}_v}\varphi_v(x^{-1}tx)dx$ and 
$\Phi(\varphi'_v,t')=\int_{\overline{Z(t')}\backslash \overline{G}'_v}\varphi'_v(x^{-1}t'x)dx$,
\end{center} 
where the Haar measures on the quotients $\overline{G}_v(t)\backslash \overline{G}_v$ and $\overline{G}'_v(t') \backslash \overline{G}'_v$ are compatibly normalized (see \cite[\S 8.B]{GJ79}). 
We say $\varphi_v$ and $\varphi'_v$ have  {\it matching orbital integrals} and write $\varphi_v\sim \varphi'_v$ if the following conditions are satisfied:
\begin{enumerate}
    \item If $v\notin S$, $\varphi_v=\varphi'_v$ via $G_v\cong G'_v$. 
    \item if $v\in S$, there are two cases: 
    \begin{enumerate}
        \item for elliptic elements $t\sim t'$ in $G_v,G_v'$ respectively, $\Phi(\varphi_v,t)=\Phi(\varphi'_v,t')$;
        \item for $t\in G_v$, $\Phi(\varphi_v,t)=0$.
    \end{enumerate} 
\end{enumerate}
We say $\varphi=\Pi_{v\in V}\varphi_v$ and $\varphi'=\Pi_{v\in V}\varphi'_v$ have (global) {\it matching orbit integral} and write $\varphi\sim \varphi'$ if $\varphi_v$ and $\varphi'_v$ have matching orbital integrals for every $v\in V$. 

Let $R_{\omega,0}$ be the representation of $G(\mathbb{A})$ on $L_{\rm{cusp}}^2(Z(\mathbb{A})G(F)\backslash G(\mathbb{A}),\omega)$ and $R'_{\omega,0}$ be the representation of $D^{\times}(\mathbb{A})$ on $L^2(Z(\mathbb{A})D^{\times}(F)\backslash D^{\times}(\mathbb{A}),\omega)$. 
Following \cite[(8.16)]{GJ79}, we have the following results. 
\begin{proposition}\label{pJL}
If $\varphi=\Pi_{v\in V}\varphi_v\sim \varphi'=\Pi_{v\in V}\varphi'_v$, we have
\begin{equation}
    \tr R_{\omega,0}(\phi)=\tr R'_{\omega,0}(\phi'),
\end{equation}
and 
\begin{equation}
    \vol(Z(\mathbb{A})G(F)\backslash G(\mathbb{A}))\phi(1)=\vol(Z(\mathbb{A})D^{\times}(F)\backslash D^{\times}(\mathbb{A}))\phi'(1).
\end{equation}
\end{proposition}


We let $G_S=\prod_{v\in S}G(F_v)$ and $G'_S=\prod_{v\in S}D^{\times}(F_v)$. 
Recall that $S$ is the ramification set of $D$, i.e., $D^{\times}(F_v)\ncong \GL(2,F_v)$ if and only if $v\in S$. 
Note $\overline{G'}_v$ is compact if $v\in S$.
Thus $\overline{G'}_S$ is compact.  

For $v\in V_f$, we let $\St_v$ be the Steinberg representation of $\GL(2,F_v)$. 
For $v\in V_{\infty}$, we let $\St_v$ be the square-integrable representation of minimal formal degree, i.e., $\St_v=\omega_v\otimes (\mathcal{D}_2^+\oplus \mathcal{D}_2^-)$ where $\mathcal{D}_n^{\pm}$ are the holomorphic/anti-holomorphic discrete series of $\SL(2,\mathbb{R})$ (see \cite{Kn77GL2}). 
We have the following result (see \cite[\S 8.2]{GJ79}, \cite[\S 6]{Lldsbase80} and \cite{Rog80}).

\begin{lemma}\label{lStdimvol} 
Given $\varphi'_v \in C_{\cpt}^{\infty}(G'_v,\omega_v)$, there exists $\varphi_v \in C_{\cpt}^{\infty}(G_v,\omega_v)$ such that $\varphi_v\sim \varphi'_v$. 
Moreover, for $z\in Z(F_v)$, we have $\frac{\varphi_v(z)}{d(\St_v)}=-\varphi'_v(z)\vol(F_v^{\times}\backslash G'_v)$. 
\end{lemma}

Consider a square-integrable representation $\pi_v$ of $\GL(2,F_v)$ and an irreducible representation $\pi'_v$ of $D^{\times}(F_v)$ which are matched by the local Jacquet-Langlands correspondence, i.e., $\JL_v(\pi'_v)=\pi_v$. 
The formal degrees $d(\pi_v)$ are proportional in the following sense. 
\begin{lemma}\label{ldgdim}
For each $v$, if $\JL_v(\pi'_v)=\pi_v$, then $\frac{d(\pi_v)}{d(\St_v)}=\dim_{\mathbb{C}}\pi'_v$. 
\end{lemma}
\begin{proof}
This result is proved for $\GL(n)$ over $p$-adic field by J. Rogawski \cite{Rgw83}. 
It remains to prove for $\GL(2,\mathbb{R})$. 
We let $H_n$ be $\mathcal{D}_n^+\oplus \mathcal{D}_n^-$ (up to a central character). 
It is known that the formal degree of $H_n$ is always $c(n-1)$ for some constant $c>0$ that is uniquely determined by the Haar measure we choose on $\GL(2,\mathbb{R})$. 
On the other hand, as $D^\times(\mathbb{R})\cong \mathbb{R}^{\times}_{>0}\times \SU(2)$,  $H_n$ corresponds to $\mathbb{C}^{n-1}$, the $(n-1)$-dimensional representation of $\SU(2)$ (see \cite[\S 1]{KnRg97tfJL}). 
Thus we have $\frac{d(H_n)}{d(\St_{\infty})}=n-1$. 
\end{proof}

For $\pi=\otimes_{v\in V} \pi_v$ and $\pi'=\otimes_{v\in V} \pi'_v$, we let $\pi_S=\otimes_{v\in V} \pi_v$ and $\pi'_S=\otimes_{v\in V} \pi'_v$ be the representation of $G_S$ and $G'_S$. 

\begin{proposition}\label{pSfd}
If $\JL(\pi')=\pi$, we have
\begin{center}
    $\vol(\overline{G}(F)\backslash \overline{G}(\mathbb{A}_F))d(\pi_S)=\vol(\overline{G'}(F)\backslash \overline{G'}(\mathbb{A}_F))d(\pi'_S)$. 
\end{center}
\end{proposition}
\begin{proof}
By Proposition \ref{pJL} and Lemma \ref{lStdimvol}, we have
\begin{equation*}
    \vol(\overline{G}(F)\backslash \overline{G}(\mathbb{A}_F))\prod_{v\in S}d(\St_v)=\vol(\overline{G'}(F)\backslash \overline{G'}(\mathbb{A}_F))\prod_{v\in S}\vol(\overline{G'}_v)^{-1}. 
\end{equation*}
Observe $\dim \pi'_v=\vol(\overline{G'}_v)d(\pi'_v)$ (which can also be shown by Proposition \ref{pas} where we take $\Gamma$ to be trivial). 
We apply Lemma \ref{ldgdim} and obtain
\begin{equation*}
\begin{aligned}
&\vol(\overline{G}(F)\backslash \overline{G}(\mathbb{A}_F))d(\pi_S)\\=&
    \vol(\overline{G}(F)\backslash \overline{G}(\mathbb{A}_F))\prod_{v\in S}d(\pi_v)\\
    =&\vol(\overline{G}(F)\backslash \overline{G}(\mathbb{A}_F))\prod_{v\in S}(d(\St_v)\dim \pi'_v)\\
    =&\vol(\overline{G'}(F)\backslash \overline{G'}(\mathbb{A}_F))\prod_{v\in S}(\vol(\overline{G'}_v)^{-1}\dim \pi'_v)\\
    =&\vol(\overline{G'}(F)\backslash \overline{G'}(\mathbb{A}_F))\prod_{v\in S}d(\pi'_v)
    =\vol(\overline{G'}(F)\backslash \overline{G'}(\mathbb{A}_F))d(\pi'_S). 
\end{aligned}
\end{equation*}
\end{proof}
\begin{remark}\label{rtama1}
The two volumes $\vol(\overline{G}(F)\backslash \overline{G}(\mathbb{A}_F))$ and $\vol(\overline{G'}(F)\backslash \overline{G'}(\mathbb{A}_F))$ are both $2$ in Proposition \ref{pSfd} if we take the Tamagawa measures (see \cite[\S 3]{Vign80} and \cite{Kott88}). 
This fact is not needed here. 
The normalizations of local and global Haar measures are discussed in the following sections. 
\end{remark}

\section{Modules over twisted group von Neumann algebras}\label{svna}

Let $\Gamma$ be a countable discrete group.
Denote by $l^2(\Gamma)$ the Hilbert space of  square-integrable functions on $\Gamma$ with the canonical orthonormal basis $\{\delta_{\gamma}|\gamma\in \Gamma\}$.  
We define
\begin{enumerate}
    \item a {\it $2$-cocyle} on $\Gamma$ with values in $\mathbb{T}=\{z\in \mathbb{C}| |z|=1\}$ is a map $\sigma\colon \Gamma\times \Gamma \to \mathbb{T}$ such that $\sigma(\alpha,\beta)\sigma(\alpha\beta,\gamma)=\sigma(\beta,\gamma)\sigma(\alpha,\beta\gamma)$ and $\sigma(\alpha,e)=\sigma(e,\alpha)=1$ for $\alpha,\beta,\gamma\in \Gamma$;
    \item a {\it $\sigma$-projective representation} $\pi$ on a Hilbert space $H$ is a map $\pi\colon \Gamma\to U(H)$ such that $U(\alpha)U(\beta)=\sigma(\alpha,\beta)U(\alpha\beta)$ for $\alpha,\beta\in \Gamma$;
    \item the {\it $\sigma$-projective right regular representation} is the representation $\rho_{\sigma}$ on $l^2(\Gamma)$ given by $\rho_{\sigma}(\gamma)f(x)=\sigma(x,\gamma)f(x\gamma)$ for all $\gamma \in \Gamma$ and $f\in L^2(\Gamma)$;
    \item the {\it $\sigma$-twisted right group von Neumann algebra} $\mathcal{R}(\Gamma,\sigma)$ is the weak operator closed algebra generated by $\{\rho_{\sigma}(\gamma)|\gamma\in \Gamma\}$. 
\end{enumerate}
Note the {\it $\sigma$-projective left regular representation} $\lambda_{\sigma}$ and the {\it $\sigma$-twisted left group von Neumann algebra} $\mathcal{L}(\Gamma,\sigma)$ can be defined in a similar and obvious way. 
It is known that $\mathcal{R}(\Gamma,\overline{\sigma})$ is the commutant of $\mathcal{L}(\Gamma,\sigma)$ on $l^2(\Gamma)$, where $\overline{\sigma}$ denotes the complex conjugate of $\sigma$ (see \cite[\S 1]{Kleppner62}). 
If $\sigma$ is trivial, all these twisted and projective objects are reduced to the ordinary ones, which are usually denoted by $\mathcal{L}(\Gamma)$ and $\mathcal{R}(\Gamma)$. 

There is a natural trace $\tau\colon \mathcal{L}(\Gamma,\sigma)\to \mathbb{C}$ given by $\tau(x)=\langle x\delta_e,\delta_e\rangle_{l^2(\Gamma)}$. 
It gives an inner product on $\mathcal{L}(\Gamma,\sigma)$ defined by $\langle x,y \rangle_{\tau}=\tau(xy^*)$ for $x,y\in \mathcal{L}(\Gamma,\sigma)$. 
The completion of $\mathcal{L}(\Gamma,\sigma)$ with respect to this inner product is exactly $l^2(\Gamma)$. 

Consider a normal unital representation $\pi\colon \mathcal{L}(\Gamma,\sigma) \to B(H)$. 
It is known that there is a $\mathcal{L}(\Gamma,\sigma)$-equivariant isometry 
\begin{center}
    $u\colon H\to l^{2}(\Gamma)\times l^2(\mathbb{N})$
\end{center}
of $\mathcal{L}(\Gamma,\sigma)$-modules (see \cite[\S 8.2]{AnaPopa17}). 
Hence $uu^*$ is in the commutant, i.e.,
\begin{center}
    $uu^*\in \mathcal{L}(\Gamma,\sigma)'\cap B(l^{2}(\Gamma)\times l^2(\mathbb{N}))\cong \mathcal{R}(\Gamma,\overline{\sigma})\otimes B(l^2(\mathbb{N}))$. 
\end{center}
Recall that we have a trace $\tau\otimes\tr$ on the right-hand side. 
We define the {\it von Neumann dimension} of $H$ over $\mathcal{L}(\Gamma,\sigma)$ by
\begin{center}
    $\dim_{\mathcal{L}(\Gamma,\sigma)}H:=(\tau\otimes\tr)(uu^*)$,
\end{center}
which does not depend on the choice of isometry $u$. 
Recall that a group $\Gamma$ is called of infinite conjugacy class (or simply ICC) if the conjugacy class $C_\gamma =\{\alpha\gamma\alpha^{}-1|\alpha\in \Gamma\}$ is infinite for $\gamma\neq e$. 
The following results are well-known (see \cite{JS97}).
\begin{proposition}\label{pvdim}
\begin{enumerate}
\item $\dim_{\mathcal{L}(\Gamma,\sigma)}H_1\oplus H_2=\dim_{\mathcal{L}(\Gamma,\sigma)}H_1+\dim_{\mathcal{L}(\Gamma,\sigma)}H_2$;
    \item If $\Gamma$ is ICC, $\dim_{\mathcal{L}(\Gamma,\sigma)}H_1\cong \dim_{\mathcal{L}(\Gamma,\sigma)}H_2$ if and only if $H_1\cong H_2$ as $\mathcal{L}(\Gamma,\sigma)$-modules;
    \item If $\Gamma$ is finite, $\dim_{\mathcal{L}(\Gamma,\sigma)}H=\frac{\dim_{\mathbb{C}} H}{|\Gamma|}$. 
\end{enumerate}

\end{proposition}

Suppose $\Gamma$ is a lattice in a unimodular group $G$, i.e., $\mu(\Gamma\backslash G)$ is finite for (any) Haar measure $\mu$ of $G$. 
Let $\rho_G$ be the right regular representation of $G$ on $L^2(G,\mu)$. 
Consider its restriction $\rho_G|_{\Gamma}$ on $\Gamma$ and its twisted representation $\rho_G|_{\Gamma,\sigma}$. 
We know that 
\begin{center}
    $L^2(G)\cong  l^2(\Gamma)\otimes L^2(\Gamma\backslash G)$
\end{center}
as $\mathcal{L}(\Gamma,\sigma)$ modules (see \cite[Proposition 4.1]{Enstad22}). 
Let $H$ be a square-integrable representation of $G$. 
There exists a left $G$-invariant map $p_H\colon L^2(G)\to H$. 
As $p_H$ commutes with the left action of $\Gamma$, we known $p\in \mathcal{R}(\Gamma,\overline{\sigma})\otimes B(l^2(\mathbb{N}))$ and it gives the von Neumann dimension $\dim_{\mathcal{L}(\Gamma,\sigma)}H=(\tau\otimes \tr)(p)$. 
It is related to the formal degree $d(\pi)$ of $H$ as follows. 
\begin{proposition}\label{pas}
Let $G$ be a unimodular group and $\Gamma$ be a lattice in $G$. 
Suppose $\pi$ is a ordinary (or $\sigma$-projective) square-integrable representation of $G$. 
We have
\begin{center}
    $\dim_{\mathcal{L}(\Gamma,\sigma)}\pi=\mu(\Gamma\backslash G)\cdot d(\pi)$.
\end{center}
\end{proposition}
This result was first proved in \cite[\S 3]{AS77} for semi-simple simply-connected Lie groups without compact factors (see also \cite[\S 3.5]{GHJ}). 
It is then extended to projective representations in \cite{Ra98}.
See also \cite{Bek00} and\cite[Theorem 4.3]{Enstad22} for more details. 

We discuss the ICC condition of $S$-arithmetic groups, which is needed in the following sections. 
This is an extension of \cite[Lemma 3.3.1]{GHJ}. 
\begin{proposition}\label{pSicc}
Suppose $G$ is a semi-simple algebraic group over a totally real field $F$ such that $G(\mathbb{R})$ is connected, center-free, and without a compact factor, then $\Gamma=G(\mathcal{O}_S)$ is an ICC group if $V_{\infty}\subset S$. 
\end{proposition}
\begin{proof}
Take $\gamma\in \Gamma$ such that its conjugacy class $C_{\Gamma}(\gamma)=\{g\gamma g^{-1}|g\in \Gamma\}$ is finite. 
Then the Zariski closure $\overline{C_{\Gamma}(\gamma)}=C_{\Gamma}(\gamma)$. 
By Borel density theorem, \cite[Theorem 3.2.6]{Zim84}, $G(\mathcal{O}_F)$ is Zariski dense in $G(\mathbb{R})$.
Thus $\Gamma$ is also Zariski dense. 
Consider the map $G(\mathbb{R})\to \overline{C_{\Gamma}(\gamma)}$ given by $g\mapsto g\gamma g^{-1}$. 
We have $|G(\mathbb{R})/Z_{G(\mathbb{R})}(\gamma)|=|\overline{C_{\Gamma}(\gamma)}|$ is finite. 
Thus $Z_{G(\mathbb{R})}(\gamma)$ is a finite-index subgroup of $G(\mathbb{R})$. 
Then $Z_{G(\mathbb{R})}(\gamma)=G(\mathbb{R})$ and $\gamma\in Z(G(\mathbb{R}))$, which shows $\gamma=e$, the identity. 
\end{proof}

\begin{remark}
If $H^2(\Gamma;\mathbb{T})$ is trivial, e.g., $H^2(\PSL(2,\mathbb{Z});\mathbb{T})=1$, all these $\mathcal{L}(\Gamma,\sigma)$ are isomorphic to the untwisted group von Neumann algebra $\mathcal{L}(\Gamma)$.
For a $S$-arithmetic groups, its second cohomology groups are not trivial in general (see \cite{AdemNaf98}). 
Their twisted group von Neumann algebras may be not isomorphic to the ordinary one. 
\end{remark}

\section{Dimensions over the principal $S$-arithmetic groups}\label{sdim}

Let $F$ be a totally real number field with degree $[F:\mathbb{Q}]=n$. 
For each non-archimedean place $v$, we denote the integral ring by $\mathcal{O}_v$ and the maximal ideal by $\mathfrak{m}_v$. 
Let $\varpi_v$ be a uniformizer. 
Let $q_v$ the cardinality of the residue field  $\mathcal{O}_v/\mathfrak{m}_v\cong\mathbb{F}_{q_v}$. 
For the finite set of ramified places $S$ such that $V_{\infty}\subset S $, we let $\mathcal{O}_S$ be the ring of $S$-integers, i.e., $\mathcal{O}_S=\{x\in F^{*}||x|_v\leq  1,v\notin S\}$. 

We will consider the correspondence over the $S$-arithmetic groups of $\PGL(2,\mathcal{O}_S),\SL(2,\mathcal{O}_S)$ as well as the ones of $\PD^{\times}(\mathcal{O}_S),D^{1}(\mathcal{O}_S)$.
Note the latter two are finite groups. 
We first discuss $\SL(2)$ with a family of local Haar measures $\{\nu_p\}$ given in \cite[\S 3.5-3.6]{Pras90}:
\begin{itemize}
    \item For $p\in S\cap V_{\infty}$, $F_v=\mathbb{R}$ and the Haar measure is normalized in the sense $\nu_p(\SO(2))=1$. 
    \item For $p\in S\cap V_f$, let $I_v$ be a Iwahori subgroup of $\SL(2,F_v)$, e.g, the inverse image of $\{\begin{pmatrix}
        a & b\\
        0 & a^{-1}
    \end{pmatrix}|a\in \mathbb{F}_{q_v}^*, b\in \mathbb{F}_{q_v}\}$ under the map $\SL(2,\mathcal{O}_v)\to \SL(2,\mathbb{F}_{q_v})$. 
    The Haar measure is normalized by $\nu_v(I_v)=1$. 
    As $|\SL(2,\mathcal{O}_v)/I_v|=q_v+1$, $\nu_v(\SL(2,\mathcal{O}_v))=q_v+1$. 
\end{itemize}
Let $\nu_S=\prod_{v\in S}\nu_v$ be the product measure on $\SL(2,\mathbb{A}_S)=\prod_{v\in S}\SL(2,F_v)$.  
Thus, by \cite[Theorem 3.7]{Prasad89}, we have
\begin{equation}\label{eSL2S}
\nu_S(\SL(2,\mathcal{O}_S)\backslash \SL(2,\mathbb{A}_S))=d_F^{3/2} \cdot \frac{1}{(2\pi)^{2n}}\cdot \zeta_F(2)\cdot \prod_{v\in S_f}(q_v+1),
\end{equation}
where $d_F$ is the discriminant of $F$ and $\zeta_F$ is the Dedekind zeta function of $F$. 
The functional equation of $\zeta_F$ (see \cite[Theorem VII.5.9]{NeuANT}) gives $\zeta_F(2)=\frac{(2\pi)^{2n}}{2^n}d_F^{-3/2}|\zeta_F(-1)|$. 
Thus we obtain
\begin{equation}
\nu_S(\SL(2,\mathcal{O}_S)\backslash \SL(2,\mathbb{A}_S)) =\frac{|\zeta_F(-1)|}{2^n}\prod_{v\in S_f}(q_v+1),
\end{equation}
where $\zeta_F(-1)$ is proved to be rational \cite[Theorem 8.1]{DeRib}.

\begin{lemma}\label{lratioPGLPSL}
Let $R$ be a unital commutative ring. 
We have
\begin{center}
$\PGL(2,R)/\PSL(2,R)\cong R^{*}/(R^*)^2$. 
\end{center}
\end{lemma}
\begin{proof}
Let us consider the determinant $\det \colon \GL(2,R)\to R^*$. 
It induces a homomorphism $\eta \colon \PGL(2,R) \to R^*/(R^*)^2$ given by
\begin{center}
    $\eta([g])=[\det(g)]\in R^*/(R^*)^2$,
\end{center}
where $[g]$ is the image of $g\in \GL(2,R)$ in $\PGL(2,R)$. 
If $[g_1]=[g_2]$, there exists some $a \in R^*$ such that $g_1=a g_2$ and $\det(g_1)=a^2\det(g_2)$. 
Hence $\eta$ is a well-defined map.
As $\ker \eta=\PSL(2,R)$, 
we have $\PGL(2,R)/\PSL(2,R)\cong R_v^*/(R_v^*)^2$.
\end{proof}

For a prime number $p$, we let $\delta_F(p,S)=\sum_{v\in S,v|p} e_v f_v$, where $e_v,f_v$ are the ramification index and inertia degree of a place $v$ over $p$.    
\begin{lemma}\label{lmuPGLS}
With respect to the Haar measures above, we have
\begin{center}
   $\mu_S(\PGL(2,\mathcal{O}_S)\backslash \PGL(2,\mathbb{A}_S))=\frac{2^{\delta_F(2,S)+1}\cdot |\zeta_F(-1)|}{2^{2n}}\prod_{v\in S_f}(q_v+1)$
\end{center}
\end{lemma}
\begin{proof}
We first identify $\frac{\mu_S(\PGL(2,\mathcal{O}_S)\backslash \PGL(2,\mathbb{A}_S))}{\nu_S(\SL(2,\mathcal{O}_S)\backslash \SL(2,\mathbb{A}_S))}$ with
\begin{center}
    $\prod_{v\in S}\frac{|\PGL(2,F_v)/\PSL(2,F_v)|}{|\SL(2,F_v)/ \PSL(2,F_v)|} \cdot \left( \frac{|\PGL(2,\mathcal{O}_S)/\PSL(2,\mathcal{O}_S|)}{|\SL(2,\mathcal{O}_S)/\PSL(2,\mathcal{O}_S)|}\right)^{-1}$. 
\end{center}
Observe $|F_v^*/(F_v^*)^2|=2,4,2^{e_v f_v}$ if $v$ lies over $\mathbb{R}$, odd prime or $2$ (see \cite[Proposition 4.3.20]{CohenNT1}).  
We apply Lemma \ref{lratioPGLPSL} to get $|\PGL(2,F_v)/\PSL(2,F_v)|$ and the Dirichlet's Theorem for $S$-units \cite[\S 5.4]{PlaVla94} to get $|\PGL(2,\mathcal{O}_S)/\PSL(2,\mathcal{O}_S|)$. 
The ratio above is $2^{\delta_F(2,S)+1-n}$. 
Then it follows Equation \ref{eSL2S}. 
\end{proof}

Now we consider the formal degrees of the local Steinberg representations 
$\St_p$ of $\GL(2,F_v)$ and $\St_S=\prod_{v\in S}\St_v$ of $\GL(2,\mathbb{A}_S)$. 
\begin{lemma}\label{ldSTS}
With respect to the Haar measures above, we have
\begin{center}
    $d(\St_S)=2^n\cdot 2^{-\delta_{F}(2,S)}\cdot \prod_{v\in S_f}\frac{q_v-1}{2(q_{v}+1)}$. 
\end{center}
\end{lemma}
\begin{proof}
The formal degrees of $\St_v$ are listed as follows. 
\begin{enumerate}
    \item For $v\in V_{\infty}$, 
    the Haar measure is normalized as $\mu_{\infty}(\SL(2,\mathbb{Z})\backslash \SL(2,\mathbb{R}))=\frac{1}{24}$, we have $\mu_{\infty}(\PSL(2,\mathbb{Z})\backslash \PSL(2,\mathbb{R}))=\frac{1}{24}$. 
    As shown in
    \cite[Example 3.3.4]{GHJ}, $\dim_{\mathcal{L}(\PSL(2,\mathbb{Z}))}\mathcal{D}_2^+=\frac{1}{12}$ (see the discussion below Lemma \ref{ldgdim} for $\mathcal{D}_2^\pm$). 
    Thus, by Proposition \ref{pas}, $d(\mathcal{D}_2^+)=d(\mathcal{D}_2^-)=2$ as discrete series of $\SL(2,\mathbb{R})$. 
    
    Note $\St_{\infty}=\mathcal{D}_2^+ \oplus \mathcal{D}_2^-$ (up to a central character). 
    Hence $\dim_{\mathcal{L}(\PSL(2,\mathbb{Z}))}\St_{\infty}=\frac{1}{6}$. 
    Observe that
    \begin{center}
        $\mu_{\infty}(\PSL(2,\mathbb{Z})\backslash \PGL(2,\mathbb{R}))=2\mu_{\infty}(\PSL(2,\mathbb{Z})\backslash \PSL(2,\mathbb{R}))=\frac{1}{12}$. 
    \end{center}
    Thus, by Proposition \ref{pas} again, we know $d(\St_{\infty})=2$. 
    
    \item For $p\in S_f$, as $\mu_v(\SL(\mathcal{O}_v))=q_{v}+1$, we have $\mu_v(\PSL(\mathcal{O}_v))=\frac{q_{v}+1}{2}$ and $\mu_v(\PGL(\mathcal{O}_v))=q_{v}+1$ or $2^{e_{v}f_{v}}(q_{v}+1)$ when $2\mid q_v$ or $2\nmid q_v$.  
    By \cite[Lemma 2.2]{Rgw83} and \cite[\S 2.2]{CMS90} (where the local Haar measure $\nu_p$ is normalized in the sense of that the volume $\PGL(\mathcal{O}_v)$ is always $1$), 
     we have $d(\St_v)=\frac{q_v-1}{2(q_{v}+1)}$ or $2^{-e_{v}f_{v}}\frac{q_v-1}{2(q_{v}+1)}$, which depends on $2\mid q_v$ or $2\nmid q_v$. 
\end{enumerate}
Hence we have
\begin{center}
    $d(\St_S)=\prod_{v\in S}d(\St_v)=2^n\cdot 2^{-\delta_{F}(2,S)}\cdot \prod_{v\in S_f}\frac{q_v-1}{2(q_{v}+1)}$. 
\end{center}
\end{proof}

We let $\sigma_{\pi}$ denote the $2$-cocycle of the projective representation which comes from an ordinary irreducible representation $\pi$. 
\begin{proposition}\label{pvdimSt}
    $\dim_{\mathcal{L}(\PGL(2,\mathcal{O}_S),\sigma_{\St_S})}\St_S=\frac{2|\zeta_F(-1)|}{2^{|S|}}\prod_{v\in S_f}(q_{v}-1)$
\end{proposition}
\begin{proof}
It is straightforward by Proposition \ref{pas}, Lemma \ref{lmuPGLS}, and \ref{ldSTS}. 
\end{proof}

\begin{theorem}\label{tvndimquotient}
Given a global Jacquet-Langlands correspondence  $\JL\colon \pi'=\otimes_{v\in V}\pi'_v$ $\mapsto$ $\pi=\otimes_{v\in V} \pi_v$, 
we have
\begin{equation*}
\frac{\dim_{\mathcal{L}(\PGL(2,\mathcal{O}_S),\sigma_{\pi_S})}\pi_S}{\dim_{\mathcal{L}(\PD^{\times}(\mathcal{O}_S),\sigma_{\pi'_S})}\pi'_S}=\frac{2|\zeta_F(-1)||\PD^{\times}(\mathcal{O}_S)|}{2^{|S|}}\prod_{v\in S_f}(q_{v}-1)
\end{equation*}
\end{theorem}
\begin{proof}
By Proposition \ref{pvdim}, $\dim_{\mathcal{L}(\PD^{\times}(\mathcal{O}_S),\sigma_{\pi'_S})}\pi'_S=\frac{\dim \pi'_S}{|\PD^{\times}(\mathcal{O}_S)|}$. 
It suffice to prove $\frac{\dim_{\mathcal{L}(\PGL(2,\mathcal{O}_S),\sigma_{\pi_S})}\pi_S}{\dim \pi'_S}=\frac{2|\zeta_F(-1)|}{2^{|S|}}\prod_{v\in S_f}(q_{v}-1)$.
By Lemma \ref{ldgdim}, it is equivalent to $\dim_{\mathcal{L}(\PGL(2,\mathcal{O}_S),\sigma_{\St_S})}\St_S=\frac{2|\zeta_F(-1)|}{2^{|S|}}\prod_{v\in S_f}(q_{v}-1)$, which is obtained in Proposition \ref{pvdimSt}. 
\end{proof}


Let $\zeta_{D}$ be the zeta function of $D$ (see \cite[\S 3.2]{Vign80}), which has a simple pole at $z=1$ and is known to have a factorization:
\begin{center}
$\zeta_{D}(s)=\zeta_{F}(2s)\zeta_{F}(2s-1)\prod_{v\in S_f}(1-q_{v}^{1-2s})$.   
\end{center}
Observe $|\SL(2,\mathcal{O}_S)/\PSL(2,\mathcal{O}_S)|=2$ and $|\PGL(2,\mathcal{O}_S)/\PSL(2,\mathcal{O}_S)|=2^{|S|}$ by Lemma \ref{lratioPGLPSL} together with Dirichlet's Theorem for $S$-units \cite[\S 5.4]{PlaVla94}.  
We may take the discrete subgroups $\SL(2,\mathcal{O}_S))$ and $D^1(\mathcal{O}_S)$ in $D^1(\mathbb{A}_1)=\prod_{v\in S}D^1(F_v)$ and obtain the following result. 
\begin{corollary}
Assume $\pi'=\otimes_{v\in V}\pi'_v \mapsto \pi=\otimes_{v\in V} \pi_v$ as above, we have
\begin{equation*}
\frac{\dim_{\mathcal{L}(\SL(2,\mathcal{O}_S))}\pi_S}{\dim_{\mathbb{C}}\pi'_S}=\abs*{ \frac{\zeta_{D}(0)}{\zeta_{F}(0)}}. 
\end{equation*}
\end{corollary} 
Please note $\SL(2,\mathcal{O}_S)$ is not center-free and thus not ICC while $\PSL(2,\mathcal{O}_S)$ is an ICC group by Proposition \ref{pSicc}. 

\begin{remark}
$\PD^{\times}(\mathcal{O}_S)$ is contained in a cyclic, dihedral or exceptional group $A_4,S_4,A_5$ (see \cite[\S 32.4-\S 32.7]{Voight21}).  
A necessary condition that $\PD^{\times}(\mathcal{O}_S)$ is a cyclic group of order of $m$ or a dihedral group of order $2m$ is that $2\cos\frac{2\pi}{m}\in F$. 
Hence there exists a constant $c_F>0$, which is uniquely determined by $F$ such that $|\PD^{\times}(\mathcal{O}_S)|< c_F$ 
for any quaternion algebra $D$ defined on $F$. 
\end{remark}

\appendix

\bibliographystyle{abbrv}
\typeout{}
\bibliography{MyLibrary} 

\begin{thebibliography}{10}

\bibitem{AdemNaf98}
A.~Adem and N.~Naffah.
\newblock On the cohomology of {${\rm SL}_2(\bold Z[1/p])$}.
\newblock In {\em Geometry and cohomology in group theory ({D}urham, 1994)}, volume 252 of {\em London Math. Soc. Lecture Note Ser.}, pages 1--9. Cambridge Univ. Press, Cambridge, 1998.

\bibitem{AnaPopa17}
C.~Anantharaman and S.~Popa.
\newblock An introduction to ii1 factors.
\newblock {\em preprint}, 8, 2017.

\bibitem{AS77}
M.~Atiyah and W.~Schmid.
\newblock A geometric construction of the discrete series for semisimple {L}ie groups.
\newblock {\em Invent. Math.}, 42:1--62, 1977.

\bibitem{Badu08}
A.~I. Badulescu.
\newblock Global {J}acquet-{L}anglands correspondence, multiplicity one and classification of automorphic representations.
\newblock {\em Invent. Math.}, 172(2):383--438, 2008.
\newblock With an appendix by Neven Grbac.

\bibitem{Bek00}
M.~B. Bekka.
\newblock Square integrable representations, lattices and von {N}eumann algebras.
\newblock In {\em Lie theory and its applications in physics, {III} ({C}lausthal, 1999)}, pages 27--40. World Sci. Publ., River Edge, NJ, 2000.

\bibitem{CohenNT1}
H.~Cohen.
\newblock {\em Number theory. {V}ol. {I}. {T}ools and {D}iophantine equations}, volume 239 of {\em Graduate Texts in Mathematics}.
\newblock Springer, New York, 2007.

\bibitem{CMS90}
L.~Corwin, A.~Moy, and P.~J. Sally, Jr.
\newblock Degrees and formal degrees for division algebras and {${\rm GL}_n$} over a {$p$}-adic field.
\newblock {\em Pacific J. Math.}, 141(1):21--45, 1990.

\bibitem{DKV84}
P.~Deligne, D.~Kazhdan, and M.-F. Vign\'{e}ras.
\newblock Repr\'{e}sentations des alg\`ebres centrales simples {$p$}-adiques.
\newblock In {\em Representations of reductive groups over a local field}, Travaux en Cours, pages 33--117. Hermann, Paris, 1984.

\bibitem{DeRib}
P.~Deligne and K.~A. Ribet.
\newblock Values of abelian {$L$}-functions at negative integers over totally real fields.
\newblock {\em Invent. Math.}, 59(3):227--286, 1980.

\bibitem{Delorme96}
P.~Delorme.
\newblock Infinitesimal character and distribution character of representations of reductive {L}ie groups.
\newblock In {\em Representation theory and automorphic forms ({E}dinburgh, 1996)}, volume~61 of {\em Proc. Sympos. Pure Math.}, pages 73--81. Amer. Math. Soc., Providence, RI, 1997.

\bibitem{Enstad22}
U.~Enstad.
\newblock The density theorem for projective representations via twisted group von {N}eumann algebras.
\newblock {\em J. Math. Anal. Appl.}, 511(2):Paper No. 126072, 25, 2022.

\bibitem{GJ79}
S.~Gelbart and H.~Jacquet.
\newblock Forms of {${\rm GL}(2)$} from the analytic point of view.
\newblock In {\em Automorphic forms, representations and {$L$}-functions ({P}roc. {S}ympos. {P}ure {M}ath., {O}regon {S}tate {U}niv., {C}orvallis, {O}re., 1977), {P}art 1}, Proc. Sympos. Pure Math., XXXIII, pages 213--251. Amer. Math. Soc., Providence, R.I., 1979.

\bibitem{GHJ}
F.~M. Goodman, P.~de~la Harpe, and V.~F.~R. Jones.
\newblock {\em Coxeter graphs and towers of algebras}, volume~14 of {\em Mathematical Sciences Research Institute Publications}.
\newblock Springer-Verlag, New York, 1989.

\bibitem{JL70}
H.~Jacquet and R.~P. Langlands.
\newblock {\em Automorphic forms on {${\rm GL}(2)$}}.
\newblock Lecture Notes in Mathematics, Vol. 114. Springer-Verlag, Berlin-New York, 1970.

\bibitem{JS97}
V.~Jones and V.~S. Sunder.
\newblock {\em Introduction to subfactors}, volume 234 of {\em London Mathematical Society Lecture Note Series}.
\newblock Cambridge University Press, Cambridge, 1997.

\bibitem{Kleppner62}
A.~Kleppner.
\newblock The structure of some induced representations.
\newblock {\em Duke Math. J.}, 29:555--572, 1962.

\bibitem{Kn77GL2}
A.~W. Knapp.
\newblock Representations of {${\rm GL}_{2}({\bf R})$} and {${\rm GL}_{2}({\bf C})$}.
\newblock In {\em Automorphic forms, representations and {$L$}-functions ({P}roc. {S}ympos. {P}ure {M}ath., {O}regon {S}tate {U}niv., {C}orvallis, {O}re., 1977), {P}art 1}, Proc. Sympos. Pure Math., XXXIII, pages 87--91. Amer. Math. Soc., Providence, R.I., 1979.

\bibitem{KnRg97tfJL}
A.~W. Knapp and J.~D. Rogawski.
\newblock Applications of the trace formula.
\newblock In {\em Representation theory and automorphic forms ({E}dinburgh, 1996)}, volume~61 of {\em Proc. Sympos. Pure Math.}, pages 413--431. Amer. Math. Soc., Providence, RI, 1997.

\bibitem{Kott88}
R.~E. Kottwitz.
\newblock Tamagawa numbers.
\newblock {\em Ann. of Math. (2)}, 127(3):629--646, 1988.

\bibitem{Lldsbase80}
R.~P. Langlands.
\newblock {\em Base change for {${\rm GL}(2)$}}, volume No. 96 of {\em Annals of Mathematics Studies}.
\newblock Princeton University Press, Princeton, NJ; University of Tokyo Press, Tokyo, 1980.

\bibitem{McRd219}
C.~Maclachlan and A.~W. Reid.
\newblock {\em The arithmetic of hyperbolic 3-manifolds}, volume 219 of {\em Graduate Texts in Mathematics}.
\newblock Springer-Verlag, New York, 2003.

\bibitem{NeuANT}
J.~Neukirch.
\newblock {\em Algebraic number theory}, volume 322 of {\em Grundlehren der mathematischen Wissenschaften [Fundamental Principles of Mathematical Sciences]}.
\newblock Springer-Verlag, Berlin, 1999.
\newblock Translated from the 1992 German original and with a note by Norbert Schappacher, With a foreword by G. Harder.

\bibitem{PlaVla94}
V.~Platonov and A.~Rapinchuk.
\newblock {\em Algebraic groups and number theory}, volume 139 of {\em Pure and Applied Mathematics}.
\newblock Academic Press, Inc., Boston, MA, 1994.
\newblock Translated from the 1991 Russian original by Rachel Rowen.

\bibitem{Prasad89}
G.~Prasad.
\newblock Volumes of {$S$}-arithmetic quotients of semi-simple groups.
\newblock {\em Inst. Hautes \'{E}tudes Sci. Publ. Math.}, (69):91--117, 1989.
\newblock With an appendix by Moshe Jarden and the author.

\bibitem{Pras90}
G.~Prasad.
\newblock Semi-simple groups and arithmetic subgroups.
\newblock In {\em Proceedings of the {I}nternational {C}ongress of {M}athematicians, {V}ol. {I}, {II} ({K}yoto, 1990)}, pages 821--832. Math. Soc. Japan, Tokyo, 1991.

\bibitem{Rog80}
J.~D. Rogawski.
\newblock An application of the building to orbital integrals.
\newblock {\em Compositio Math.}, 42(3):417--423, 1980/81.

\bibitem{Rgw83}
J.~D. Rogawski.
\newblock Representations of {${\rm GL}(n)$} and division algebras over a {$p$}-adic field.
\newblock {\em Duke Math. J.}, 50(1):161--196, 1983.

\bibitem{Ra98}
F.~R\u{a}dulescu.
\newblock The {$\Gamma$}-equivariant form of the {B}erezin quantization of the upper half plane.
\newblock {\em Mem. Amer. Math. Soc.}, 133(630):viii+70, 1998.

\bibitem{Vign80}
M.-F. Vign\'{e}ras.
\newblock {\em Arithm\'{e}tique des alg\`ebres de quaternions}, volume 800 of {\em Lecture Notes in Mathematics}.
\newblock Springer, Berlin, 1980.

\bibitem{Voight21}
J.~Voight.
\newblock {\em Quaternion algebras}, volume 288 of {\em Graduate Texts in Mathematics}.
\newblock Springer, Cham, [2021] \copyright 2021.

\bibitem{Zim84}
R.~J. Zimmer.
\newblock {\em Ergodic theory and semisimple groups}, volume~81 of {\em Monographs in Mathematics}.
\newblock Birkh\"{a}user Verlag, Basel, 1984.

\end{thebibliography}

\textit{E-mail address}: \href{mailto:junyang@fas.harvard.edu}{junyang@fas.harvard.edu}\\

{Harvard University, Cambridge, MA 02138, USA}

\end{document}